\documentclass[english]{amsart}
\usepackage[T1]{fontenc}
\usepackage[latin9]{inputenc}
\usepackage{amsthm}
\usepackage{amssymb}

\makeatletter
\numberwithin{equation}{section}
\numberwithin{figure}{section}
\theoremstyle{plain}
\newtheorem{thm}{\protect\theoremname}
\theoremstyle{plain}
\newtheorem{lem}[thm]{\protect\lemmaname}
\theoremstyle{plain}
\newtheorem{cor}[thm]{\protect\corollaryname}
\theoremstyle{remark}
\newtheorem{rem}[thm]{\protect\remarkname}

\makeatother

\usepackage{babel}
\providecommand{\corollaryname}{Corollary}
\providecommand{\lemmaname}{Lemma}
\providecommand{\remarkname}{Remark}
\providecommand{\theoremname}{Theorem}

\begin{document}
\title{An Application of Hausdorff Moment Problem}
\author{Ruiming Zhang}
\email{ruimingzhang@guet.edu.cn}
\address{School of Mathematics and Computing Sciences\\
Guilin University of Electronic Technology\\
Guilin, Guangxi 541004, P. R. China. }
\thanks{This work is supported by the National Natural Science Foundation
of China, grant No. 11771355.}
\begin{abstract}
In this work we apply Hausdorff moment problem to prove a necessary
and sufficient condition for a complex sequence to be positive. Then
we apply it to a subclass of genus $0$ entire functions $f(z)$ to
obtain an infinite family of necessary and sufficient conditions for
$f(z)$ to have only negative roots, which in turn to give infinite
many necessary and sufficient conditions for the Riemann hypothesis
and the generalized Riemann hypothesis with a primitive Dirichlet
character.
\end{abstract}

\maketitle

\section{\label{sec:Intro} Introduction}

In this work we apply Hausdorff moment problem to prove a necessary
and sufficient condition for certain complex sequence to be positive.
As an application we obtain an infinite family of necessary and sufficient
conditions for a entire function $f(z)$ to have only negative roots
where
\begin{equation}
f(z)=\sum_{n=0}^{\infty}a_{n}z^{n},\quad a_{0}>0,\ a_{n}>0,\ \forall n\in\mathbb{N},\label{eq:1.1}
\end{equation}
is of order strictly less than $1$ and the sequence of its nonzero
roots $\left\{ -\lambda_{n}\right\} _{n=1}^{\infty}\neq\emptyset$
is real-part dominating, 
\begin{equation}
\Re(\lambda_{n})\ge\beta_{0}\left|\lambda_{n}\right|,\quad\forall n\in\mathbb{N},\label{eq:1.2}
\end{equation}
where $\beta_{0}\in(0,1)$. 

The related entire function is $g(z)$, 
\begin{equation}
g(z)=\sum_{n=0}^{\infty}a_{n}(-1)^{n}z^{2n},\quad a_{0}>0,\ a_{n}>0,\ \forall n\in\mathbb{N},\label{eq:1.3}
\end{equation}
where its order is strictly less $2$. If all of its zeros are within
a horizontal strip, then $f(z)=g(i\sqrt{z})$ transforms the problem
of $g(z)$ having only real zeros to $f(z)$ with only negative zeros.
Therefore, our result can apply to $g(z)$ as well. 

It is well-known that the Riemann Xi function $\Xi(s)=\xi\left(\frac{1}{2}+is\right)$
is an even entire functions of order $1$, and all of its zeros are
within the horizontal strip $t\in\left(-\frac{1}{2},\frac{1}{2}\right)$.
The celebrated Riemann hypothesis is the assertion that all the zeros
of $\Xi(s)$ are real. Let $\chi$ be a primitive Dirichlet character.
The generalized Riemann hypothesis for $\chi$ is the claim that $\xi\left(\frac{1}{2}+is,\chi\right)$
has only real zeros. It is observed that the entire function $\xi\left(\frac{1}{2}+is,\chi\right)\cdot\xi\left(\frac{1}{2}+is,\overline{\chi}\right)$
has similar properties as $\Xi(s)$, all the zeros of $\xi\left(\frac{1}{2}+is,\chi\right)\cdot\xi\left(\frac{1}{2}+is,\overline{\chi}\right)$
are real if and only if all the zeros of $\xi\left(\frac{1}{2}+is,\chi\right)$
are real. Applying our result to $\Xi(s)$ and $\xi\left(\frac{1}{2}+is,\chi\right)\cdot\xi\left(\frac{1}{2}+is,\overline{\chi}\right)$
we formulate an infinite family of necessary and sufficient conditions
for the Riemann hypothesis and the generalized Riemann hypothesis
associated with a primitive Dirichlet character $\chi$.

\section{\label{sec:main}Main Results }
\begin{lem}
\label{lem:1} Let $\left\{ \lambda_{n}\vert n\in\mathbb{N}\right\} $
be a sequence of nonzero complex numbers such that $\exists\alpha_{0},\beta_{0}\in(0,1)$
it satisfies (\ref{eq:1.2}) and

\begin{equation}
\sum_{n=1}^{\infty}\frac{1}{\left|\lambda_{n}\right|^{\alpha_{0}}}<\infty.\label{eq:2.1}
\end{equation}
If
\begin{equation}
f(z)=\prod_{n=1}^{\infty}\left(1+\frac{z}{\lambda_{n}}\right)=\sum_{n=0}^{\infty}a_{n}z^{n},\quad a_{0}=1,\ a_{n}>0,\ \forall n\in\mathbb{N},\label{eq:2.2}
\end{equation}
then the entire function $f(z)$ is of genus $0$, the associated
heat kernel
\begin{equation}
\Theta(t\vert f)=\sum_{n=1}^{\infty}e^{-\lambda_{n}t}\label{eq:2.3}
\end{equation}
is real and in $C^{\infty}(0,\infty)$. 

For each $k\in\mathbb{N}_{0}\ \alpha,\ \beta$ satisfying $1>\alpha\ge\alpha_{0}$
and $0<\beta<\beta_{0}\inf\left\{ \left|\lambda_{n}\right|\bigg|n\in\mathbb{N}\right\} $
we have
\begin{equation}
\Theta^{(k)}(t\vert f)=\mathcal{O}\left(t^{-\alpha-k}\right),\quad t\to0^{+}\label{eq:2.4}
\end{equation}
and 
\begin{equation}
\Theta^{(k)}(t\vert f)=\mathcal{O}\left(e^{-\beta t}\right),\quad t\to+\infty.\label{eq:2.5}
\end{equation}
Furthermore, for all $k\in\mathbb{N}_{0},x\ge0$,
\begin{equation}
\int_{0}^{\infty}t^{k}e^{-xt}\left|\Theta(t\vert f)\right|dt<\infty\label{eq:2.6}
\end{equation}
and
\begin{equation}
(-1)^{k}\left(\frac{f'(x)}{f(x)}\right)^{(k)}=\sum_{n=1}^{\infty}\frac{k!}{(x+\lambda_{n})^{k+1}}=\int_{0}^{\infty}e^{-xt}t^{k}\Theta(t\vert f)dt.\label{eq:2.7}
\end{equation}
\end{lem}

\begin{proof}
Since (\ref{eq:2.1}) implies that
\[
\sum_{n=1}^{\infty}\frac{1}{\left|\lambda_{n}\right|}<\infty,
\]
then $f(z)$ is a genus $0$ entire function by definition, \cite{Boas}. 

For $a>0$ since
\[
\sup_{x>0}x^{a}e^{-x}=\left(\frac{a}{e}\right)^{a},
\]
 then for any $t>0$ and $k\ge0$, by (\ref{eq:2.1}), 
\[
\begin{aligned} & \sum_{n=1}^{\infty}\left|\lambda_{n}^{k}e^{-\lambda_{n}t}\right|=\sum_{n=1}^{\infty}\left|\lambda_{n}\right|^{k}e^{-\Re(\lambda_{n})t}\le\frac{1}{(\beta_{0}t)^{\alpha_{0}+k}}\sum_{n=1}^{\infty}\frac{\left(\beta_{0}t\left|\lambda_{n}\right|\right)^{k+\alpha_{0}}e^{-\beta_{0}\left|\lambda_{n}\right|t}}{\left|\lambda_{n}\right|^{\alpha_{0}}}\\
 & \le\frac{\sup_{x>0}x^{k+\alpha_{0}}e^{-x}}{(\beta_{0}t)^{\alpha_{0}+k}}\sum_{n=1}^{\infty}\frac{1}{\left|\lambda_{n}\right|^{\alpha_{0}}}=\left(\frac{k+\alpha_{0}}{e\beta_{0}t}\right)^{k+\alpha_{0}}\sum_{n=1}^{\infty}\frac{1}{\left|\lambda_{n}\right|^{\alpha_{0}}}<\infty,
\end{aligned}
\]
which proves that $\Theta(t\vert f)\in C^{\infty}(0,\infty)$ and
(\ref{eq:2.4}). Since $f(z)$ has positive coefficients, all possible
non-real zeros of $f(z)$ must appear in conjugate pairs, then $\Theta(t\vert f)$
must be a real function.

Let $0<\beta<\beta_{0}\inf\left\{ \left|\lambda_{n}\right|\bigg|n\in\mathbb{N}\right\} $,
then there exists a positive number $\epsilon$ with $0<\epsilon<1$
such that 
\[
\beta=\epsilon\beta_{0}\inf\left\{ \left|\lambda_{n}\right|\bigg|n\in\mathbb{N}\right\} \le\epsilon\beta_{0}\left|\lambda_{n}\right|,\quad\forall n\in\mathbb{N},
\]
it implies that for $t\ge1$,
\begin{align*}
 & \left|e^{\beta t}\Theta^{(k)}(t\vert f)\right|\le\sum_{n=1}^{\infty}\left|\lambda_{n}\right|^{k}e^{-(\beta_{0}\left|\lambda_{n}\right|-\beta)t}\le\sum_{n=1}^{\infty}\left|\lambda_{n}\right|^{k}e^{-(1-\epsilon)\beta_{0}\left|\lambda_{n}\right|t}\\
 & \le\sum_{n=1}^{\infty}\left|\lambda_{n}\right|^{k}e^{-(1-\epsilon)\beta_{0}\left|\lambda_{n}\right|}<\infty,
\end{align*}
which establishes (\ref{eq:2.5}).

For all $x\ge0$ and $k\in\mathbb{N}_{0}$ since
\[
\begin{aligned} & \int_{0}^{\infty}t^{k}e^{-xt}\left|\Theta(t\vert f)\right|dt\le\int_{0}^{\infty}t^{k}\left(\sum_{n=1}^{\infty}e^{-(x+\Re(\lambda_{n}))t}\right)dt\le\sum_{n=1}^{\infty}\int_{0}^{\infty}t^{k}e^{-(x+\beta_{0}\left|\lambda_{n}\right|)t}dt\\
 & \le\sum_{n=1}^{\infty}\frac{k!}{\left(x+\beta_{0}\left|\lambda_{n}\right|\right)^{k+1}}\le\frac{k!}{\beta_{0}^{k+1}}\sum_{n=1}^{\infty}\frac{1}{\left|\lambda_{n}\right|^{k+1}}<\infty,
\end{aligned}
\]
then 
\[
\left(\frac{f'(x)}{f(x)}\right)^{(k)}=\sum_{n=1}^{\infty}\frac{(-1)^{k}k!}{(x+\lambda_{n})^{k+1}}=(-1)^{k}\sum_{n=1}^{\infty}\int_{0}^{\infty}t^{k}e^{-(x+\lambda_{n})t}dt=\int_{0}^{\infty}e^{-xt}t^{k}\Theta(t\vert f)dt.
\]
\end{proof}
\begin{thm}
\label{thm:2}Assume that $\left\{ \lambda_{n}\right\} _{n=1}^{\infty}$
and $f(z)$ are defined as in Lemma \ref{lem:1}. Let $\left\{ m_{k}\right\} _{k=0}^{\infty}$
be defined by
\begin{equation}
\frac{f'(0)}{z}-\frac{f'(z)}{zf(z)}=\sum_{k=0}^{\infty}m_{k}(-z)^{k},\quad|z|<1,\label{eq:2.8}
\end{equation}
then
\begin{equation}
\begin{alignedat}{1} & m_{k}=\frac{(-1)^{k}}{k!}\frac{d^{k}}{dx^{k}}\left(\frac{f'(0)}{x}-\frac{f'(x)}{xf(x)}\right)_{x=0}\\
 & =\frac{(-1)^{\ell+1}}{(\ell+1)!}\left(\frac{f'(x)}{f(x)}\right)^{(\ell+1)}\bigg|_{x=0}=\sum_{n=1}^{\infty}\frac{1}{\lambda_{n}^{k+2}}.
\end{alignedat}
\label{eq:2.9}
\end{equation}
 If 
\begin{equation}
\gamma_{0}=\inf\left\{ \Re(\lambda_{n})\bigg|n\in\mathbb{N}\right\} >1,\label{eq:2.10}
\end{equation}
then the sequence $\left\{ \lambda_{n}\right\} _{n=1}^{\infty}$ is
positive if and only if 
\begin{equation}
(-1)^{k}\Delta^{k}m_{n}=\sum_{j=1}^{\infty}\frac{1}{\lambda_{j}^{n+2}}\left(1-\frac{1}{\lambda_{j}}\right)^{k}\ge0,\quad\forall n,k\in\mathbb{N}_{0}.\label{eq:2.11}
\end{equation}
Furthermore, for any nonnegative integer $\ell$,
\begin{equation}
(-1)^{\ell}m_{\ell}=\det\begin{pmatrix}1 & 0 & 0 & 0 & \dots\dots & a_{1}^{2}-2a_{2}\\
a_{1} & 1 & 0 & 0 & \dots\dots & a_{1}a_{2}-3a_{3}\\
a_{2} & a_{1} & 1 & 0 & \dots\dots & a_{1}a_{3}-4a_{4}\\
a_{3} & a_{2} & a_{1} & 1 & \dots\dots & a_{1}a_{4}-5a_{5}\\
\vdots & \vdots & \vdots & \vdots & \ddots & \vdots\\
a_{\ell} & a_{\ell-1} & a_{\ell-2} & a_{\ell-3} & \dots & a_{1}a_{\ell+1}-(\ell+2)a_{\ell+2}
\end{pmatrix}\label{eq:2.12}
\end{equation}
and 
\begin{equation}
(-1)^{\ell}m_{\ell}=a_{1}a_{\ell+1}-(\ell+2)a_{\ell+2}-\sum_{k=1}^{\ell}m_{\ell-k}(-1)^{\ell-k}a_{k}.\label{eq:2.13}
\end{equation}
\end{thm}

\begin{proof}
Necessity of (\ref{eq:2.11}). Let $\left\{ \lambda_{n}\vert n\in\mathbb{N}\right\} $
be positive. Without losing any generality let
\[
1<\gamma_{0}\le\lambda_{1}\le\lambda_{2}\le\dots.
\]
Then,
\[
\begin{aligned} & \int_{t}^{\infty}\Theta(y\vert f)dy=\int_{t}^{\infty}\left(\sum_{n=1}^{\infty}e^{-\lambda_{n}x}\right)dy=\sum_{n=1}^{\infty}\frac{e^{-\lambda_{n}t}}{\lambda_{n}}\\
 & =\int_{0}^{\infty}e^{-yt}\frac{d\mu(y)}{y}=\int_{\gamma}^{\infty}e^{-yt}\frac{d\mu(y)}{y},
\end{aligned}
\]
where $\gamma$ is any positive number such that $1<\gamma<\gamma_{0}$
and $d\mu(y)$ is the counting measure with a unit jump at each $\lambda_{n},\ n\in\mathbb{N}$.

Then for any $x>0$ by (\ref{eq:2.7})
\begin{align*}
 & \frac{f'(x)}{f(x)}=-\int_{0}^{\infty}e^{-xt}d_{t}\left(\int_{t}^{\infty}\Theta(y\vert f)dy\right)dt\\
 & =-e^{-xt}\left(\int_{t}^{\infty}\Theta(y\vert f)dy\right)\bigg|_{t=0}^{\infty}-x\int_{0}^{\infty}e^{-xt}\left(\int_{0}^{\infty}e^{-yt}\frac{d\mu(y)}{y}\right)dt\\
 & =\int_{0}^{\infty}\Theta(y\vert f)dy-x\int_{0}^{\infty}\left(\int_{0}^{\infty}e^{-(x+y)t}dt\right)\frac{d\mu(y)}{y}\\
 & =f'(0)-x\int_{0}^{\infty}\frac{1}{x+y}\frac{d\mu(y)}{y}=f'(0)-x\int_{0}^{1/\gamma}\frac{y^{2}d\mu(y^{-1})}{1+xy},
\end{align*}
 which leads to 
\begin{equation}
\frac{f'(0)}{x}-\frac{f'(x)}{xf(x)}=\int_{0}^{1/\gamma}\frac{y^{2}d\mu(y^{-1})}{1+xy}=\int_{0}^{1}\frac{y^{2}d\mu(y^{-1})}{1+xy}.\label{eq:2.14}
\end{equation}
 Since for $\forall k\in\mathbb{N}_{0}$ and $\forall x>0$, 
\begin{align*}
 & \int_{0}^{1}\frac{y^{k+2}d\mu(y^{-1})}{(1+xy)^{k+1}}=\int_{0}^{1/\gamma}\frac{y^{k+2}d\mu(y^{-1})}{(1+xy)^{k+1}}=\int_{0}^{1/\gamma}\frac{yd\mu(y^{-1})}{(x+y^{-1})^{k+1}}\\
 & =\int_{\gamma}^{\infty}\frac{d\mu(y)}{(x+y)^{k+1}y}\le\sum_{n=1}^{\infty}\frac{1}{(x+\lambda_{n})^{k+1}\lambda_{n}}\le\sum_{n=1}^{\infty}\frac{1}{\lambda_{n}^{k+2}}<\infty,
\end{align*}
then it is clear that as $x\downarrow0$ we have $\frac{1}{(x+y)^{k+1}y}\uparrow\frac{1}{y^{k+2}}$
on $[\gamma,\infty)$. By the monotone convergence theorem,\cite{RoydenFitzpatrick}
\begin{align*}
 & m_{k}=\int_{0}^{1}y^{k+2}d\mu(y^{-1})=\int_{0}^{1/\gamma}y^{k+2}d\mu(y^{-1})=\int_{\gamma}^{\infty}\frac{d\mu(y)}{y^{k+2}}\\
 & =\int_{\gamma}^{\infty}\lim_{x\downarrow0}\frac{d\mu(y)}{(x+y)^{k+1}y}=\sum_{n=1}^{\infty}\frac{1}{\lambda_{n}^{k+2}}=\frac{(-1)^{k}}{k!}\frac{d^{k}}{dx^{k}}\left(\frac{f'(0)}{x}-\frac{f'(x)}{xf(x)}\right)_{x=0}
\end{align*}
 and
\[
(-1)^{k}\Delta^{k}m_{n}=\int_{0}^{1}y^{n+2}(1-y)^{k}d\mu(y^{-1})=\sum_{j=1}^{\infty}\frac{1}{\lambda_{j}^{n+2}}\left(1-\frac{1}{\lambda_{j}}\right)^{k}>0.
\]
Sufficiency of (\ref{eq:2.11}). Assume (\ref{eq:2.11}), then by
the Hausdorff moment problem on the real line there exists a unique
bounded nonnegative measure $d\nu(x)$ on $[0,1]$ such that \cite{Ismail,ShohatTamarkin}
\[
m_{k}=\frac{(-1)^{k}}{k!}\frac{d^{k}}{dx^{k}}\left(\frac{f'(0)}{x}-\frac{f'(x)}{xf(x)}\right)_{x=0}=\int_{0}^{1}y^{k}d\nu(y),\quad k\ge0.
\]
Since $\left(f'(0)-\frac{f'(x)}{f(x)}\right)\bigg/x$ is analytic
in $|x|<1$, then for any $t\in(0,1)$, 
\begin{align*}
 & \frac{f'(0)}{t}-\frac{f'(t)}{tf(t)}=\sum_{k=0}^{\infty}\frac{d^{k}}{dx^{k}}\left(\frac{f'(0)}{x}-\frac{f'(x)}{xf(x)}\right)_{x=0}\frac{t^{k}}{k!}=\sum_{k=0}^{\infty}m_{k}(-t)^{k}\\
 & =\sum_{k=0}^{\infty}\int_{0}^{1}(-ty)^{k}d\nu(y)=\int_{0}^{1}\sum_{k=0}^{\infty}(-ty)^{k}d\nu(y)=\int_{0}^{1}\frac{d\nu(y)}{1+yt}=\int_{1}^{\infty}\frac{yd\nu(y^{-1})}{y+t},
\end{align*}
 which gives
\begin{equation}
f'(0)-\frac{f'(t)}{f(t)}=t\int_{1}^{\infty}\frac{yd\nu(y^{-1})}{y+t}\ge0,\quad0<t<1.\label{eq:2.15}
\end{equation}
 and (\ref{eq:2.8}) is obtained by analytic continuation.

Since for any $z$ with $\Re(z)>0$,
\[
\int_{1}^{\infty}\frac{yd\nu(y^{-1})}{\left|y+z\right|^{2}}\le\int_{1}^{\infty}\frac{yd\nu(y^{-1})}{(y+\Re(z))^{2}}\le\int_{1}^{\infty}\frac{d\nu(y^{-1})}{y}=\int_{0}^{1}yd\nu(y)=m_{1}<\infty,
\]
then $\int_{1}^{\infty}\frac{yd\nu(y^{-1})}{y+z}$ , hence $z\int_{1}^{\infty}\frac{yd\nu(y^{-1})}{y+z}$
is analytic in $\Re(z)>0$. On the other hand,
\[
\frac{f'(z)}{f(z)}=\sum_{n=1}^{\infty}\frac{1}{z+\lambda_{n}}
\]
is analytic in $\Re(z)>0$, so is $\frac{f'(0)}{z}-\frac{f'(z)}{zf(z)}.$
Since according to (\ref{eq:2.15}) these two analytic functions are
equal on $z\in(0,1)$, then
\begin{equation}
f'(0)-\frac{f'(z)}{f(z)}=z\int_{1}^{\infty}\frac{yd\nu(y^{-1})}{y+z}\label{eq:2.16}
\end{equation}
in $\Re(z)>0$ by analytic continuation. 

Observe that for any $t>0$,
\[
\begin{aligned} & \left|f'(0)-\frac{f'(t)}{f(t)}\right|\le\left|f'(0)\right|+\sum_{n=1}^{\infty}\frac{1}{\left|t+\lambda_{n}\right|}\le\left|f'(0)\right|+\sum_{n=1}^{\infty}\frac{1}{t+\Re(\lambda_{n})}\\
 & \le\left|f'(0)\right|+\sum_{n=1}^{\infty}\frac{1}{\Re(\lambda_{n})}\le\left|f'(0)\right|+\frac{1}{\beta_{0}}\sum_{n=1}^{\infty}\frac{1}{\left|\lambda_{n}\right|}<\infty,
\end{aligned}
\]
 then apply Fatou's lemma to (\ref{eq:2.16}), \cite{RoydenFitzpatrick}
\[
\begin{aligned} & \int_{1}^{\infty}yd\nu(y^{-1})=\int_{1}^{\infty}\lim_{t\to+\infty}\frac{t}{y+t}yd\nu(y^{-1})=\int_{1}^{\infty}\liminf_{t\to+\infty}\frac{t}{y+t}yd\nu(y^{-1})\\
 & \le\liminf_{t\to+\infty}\int_{1}^{\infty}\frac{t}{y+t}yd\nu(y^{-1})=\liminf_{t\to+\infty}\left(f'(0)-\frac{f'(t)}{f(t)}\right)\\
 & \le\left|f'(0)\right|+\frac{1}{\beta_{0}}\sum_{n=1}^{\infty}\frac{1}{\left|\lambda_{n}\right|}<\infty.
\end{aligned}
\]
If $z_{0}\not\in(-\infty,-1]$, then 
\[
d=\inf\left\{ \left|x+z_{0}\right|:x\in[1,\infty)\right\} >0.
\]
Then for any $\left|z-z_{0}\right|<\frac{d}{2}$ and $y\ge1$ we have
\[
\left|y+z\right|\ge\left|y+z_{0}\right|-\left|z-z_{0}\right|\ge\frac{d}{2}.
\]
Hence,
\begin{align*}
\int_{1}^{\infty}\frac{yd\nu(y^{-1})}{\left|y+z\right|} & \le\frac{2}{d}\int_{1}^{\infty}yd\nu(y^{-1})<\infty,\\
\int_{1}^{\infty}\frac{yd\nu(y^{-1})}{\left|y+z\right|^{2}} & \le\frac{4}{d^{2}}\int_{1}^{\infty}yd\nu(y^{-1})<\infty.
\end{align*}
Thus $z\int_{1}^{\infty}\frac{yd\nu(y^{-1})}{y+z}$ is also analytic
on $\mathbb{C}\backslash(-\infty,-1]$. By Lebesgue's dominated convergence
theorem the following limit exists,
\[
\lim_{z\to z_{0}}\int_{1}^{\infty}\frac{yd\nu(y^{-1})}{y+z}=\int_{1}^{\infty}\frac{yd\nu(y^{-1})}{y+z_{0}}.
\]
Since the meromorphic function 
\[
f'(0)-\frac{f'(z)}{f(z)}=f'(0)-\sum_{n=1}^{\infty}\frac{1}{z+\lambda_{n}}
\]
is analytic on $\mathbb{C}\backslash\left\{ -\lambda_{n}\right\} _{n=1}^{\infty}$,
then (\ref{eq:2.16}) must hold on $\mathbb{C}\backslash\left\{ (-\infty,-1]\cup\left\{ -\lambda_{n}\right\} _{n=1}^{\infty}\right\} $
by analytic continuation. 

If there is a positive integer $n_{0}$ such that $z_{0}=-\lambda_{n_{0}}$
is not negative. Since $z_{0}=-\lambda_{n_{0}}$ is a zero of $f(z)$,
then it is a simple pole of $\frac{f'(z)}{f(z)}$ with positive residue.
Hence,
\[
\lim_{z\to-\lambda_{n_{0}}}\left(f'(0)-\frac{f'(z)}{f(z)}\right)=\infty.
\]
On the other hand, since $z_{0}=-\lambda_{n_{0}}$ is not in $(-\infty,-1]$,
then by (\ref{eq:2.10}) it is not a singularity of $z\int_{1}^{\infty}\frac{yd\nu(y^{-1})}{y+z}$
and
\[
\lim_{z\to-\lambda_{n_{0}}}z\int_{1}^{\infty}\frac{yd\nu(y^{-1})}{y+z}=-\lambda_{n_{0}}\int_{1}^{\infty}\frac{yd\nu(y^{-1})}{y-\lambda_{n_{0}}}
\]
is finite, which leads to a contradiction to (\ref{eq:2.16}). Therefore,
all $\left\{ \lambda_{n}\right\} _{n\in\mathbb{N}}$ must be positive. 

The second expression of $m_{\ell}$ in (\ref{eq:2.9}) is obtained
from the logarithmic derivative of the infinite product expansion
of $f(z)$ in (\ref{eq:2.17}).

Since by (\ref{eq:2.8}),
\[
f'(0)f(z)-f'(z)=z\sum_{n=0}^{\infty}\left(\sum_{k=0}^{n}m_{k}(-1)^{k}a_{n-k}\right)z^{n}=z\sum_{n=0}^{\infty}\left(a_{1}a_{n+1}-(n+2)a_{n+2}\right)z^{n},
\]
then $\forall n\ge0$,
\[
\sum_{k=0}^{n}m_{k}(-1)^{k}a_{n-k}=\sum_{k=0}^{n}m_{n-k}(-1)^{n-k}a_{k}=a_{1}a_{n+1}-(n+2)a_{n+2},
\]
which gives (\ref{eq:2.13}).

Let $n=0,1,\dots,\ell$ we get the following system,
\[
\begin{pmatrix}1 & 0 & 0 & 0 & \dots & 0\\
a_{1} & 1 & 0 & 0 & \dots & 0\\
a_{2} & a_{1} & 1 & 0 & \dots & 0\\
a_{3} & a_{2} & a_{1} & 1 & \dots & 0\\
\vdots & \vdots & \vdots & \vdots & \ddots & \vdots\\
a_{\ell} & a_{\ell-1} & a_{\ell-2} & a_{\ell-3} & \dots & 1
\end{pmatrix}\begin{pmatrix}m_{0}\\
-m_{1}\\
m_{2}\\
-m_{3}\\
\vdots\\
(-1)^{\ell}m_{\ell}
\end{pmatrix}=\begin{pmatrix}a_{1}^{2}-2a_{2}\\
a_{1}a_{2}-3a_{3}\\
a_{1}a_{3}-4a_{4}\\
a_{1}a_{4}-5a_{5}\\
\vdots\\
a_{1}a_{\ell+1}-(\ell+2)a_{\ell+2}
\end{pmatrix}.
\]
 Then (\ref{eq:2.12}) is obtained by applying Cramer's rule.
\end{proof}
\begin{cor}
\label{cor:3}Let $f(z)$ be an entire function of order strictly
less than $1$ such that for certain $\mu\in\mathbb{N}_{0}$,
\begin{equation}
z^{-\mu}f(z)=\sum_{n=0}^{\infty}a_{n}z^{n},\quad a_{0}\cdot a_{n}>0,\quad n\in\mathbb{N}.\label{eq:2.17}
\end{equation}
Assume that $f(z)$ satisfies (\ref{eq:1.2}) for certain positive
number $\beta_{0}\in(0,1)$, (\ref{eq:2.9}), (\ref{eq:2.10}) and
(\ref{eq:2.11}), then all the zeros of $f(z)$ are negative.
\end{cor}

\begin{proof}
Since the order of $f(z)$ is $\alpha_{0}\in[0,1)$, then by \cite[Theorem 2.5.18]{Boas},
the condition (\ref{eq:2.1}) holds. 

This corollary is proved by applying Theorem \ref{thm:2} to $\frac{f(z)}{a_{0}z^{\mu}}$.
\end{proof}
\begin{cor}
\label{cor:4}For certain $\mu\in\mathbb{N}_{0}$ let
\begin{equation}
z^{-2\mu}g(z)=\sum_{n=0}^{\infty}a_{n}(-z^{2})^{n},\quad a_{0}\cdot a_{n}>0,\ \forall n\in\mathbb{N}\label{eq:2.18}
\end{equation}
be an even entire function of order strictly less $2$ with nonzero
roots $\left\{ \pm z_{n}\vert n\in\mathbb{N}\right\} $ such that
there exists a positive number $M$,
\begin{equation}
\Re(z_{n})>0,\ \Re(z_{n}^{2})>1,\ \left|\Im(z_{n})\right|\le M,\quad\forall n\in\mathbb{N},\label{eq:2.19}
\end{equation}
then all the nonzero roots $\left\{ \pm z_{n}\vert n\in\mathbb{N}\right\} $
are real if and only if the sequence $\left\{ m_{k}\right\} _{k\in\mathbb{N}_{0}}$,
\begin{equation}
m_{k}=\sum_{n=1}^{\infty}\frac{1}{z_{n}^{2k+4}}=-\frac{1}{(2k+3)!\cdot2}\left(\frac{g'(z)}{g(z)}\right)^{(2k+3)}\bigg|_{z=0}\label{eq:2.20}
\end{equation}
satisfy (\ref{eq:2.11}). 
\end{cor}

\begin{proof}
Clearly, once we verify the condition (\ref{eq:1.2}) for $\lambda_{n}=z_{n}^{2}$,
the corollary is proved by applying Corollary (\ref{cor:3}) to the
entire function $f(z)=g(i\sqrt{z})$. Without losing any generality
we let 
\[
1<\Re(z_{1})\le\Re(z_{2})<\dots\le\Re(z_{n})\le\Re(z_{n+1})\le\dots.
\]
Since 
\[
\lim_{N\to\infty}\left|z_{n}\right|=+\infty,
\]
then for any $\epsilon\in(0,1)$ there exists a $N_{\epsilon}\in\mathbb{N}$
such that
\[
\epsilon\left|z_{n}\right|^{2}>2M^{2},\quad\forall n\ge N_{\epsilon}.
\]
Hence,
\[
1<\Re(z_{n}^{2})=\left|z_{n}\right|^{2}-2\left(\Im(z_{n})\right)^{2}\ge\left|z_{n}\right|^{2}-2M^{2}\ge(1-\epsilon)\left|z_{n}\right|^{2}.
\]
Let 
\[
\beta_{0}=\min\left\{ 1-\epsilon,\frac{\Re(z_{n}^{2})}{\left|z_{n}^{2}\right|},\ 1\le n\le N_{\epsilon}\right\} ,
\]
then $\beta_{0}\in(0,1)$ and
\[
\Re(z_{n}^{2})\ge\beta_{0}\left|z_{n}^{2}\right|,\quad\forall n\in\mathbb{N}.
\]
 
\end{proof}
\begin{rem}
Assume that $-\lambda_{1}$ the nonzero root of $f(z)$ with largest
real part. If $\lambda_{1}$ is not positive then $f(z)$ can not
have all negative roots. For any $L>\frac{1}{\lambda_{1}}$ the entire
function $f_{1}(z)=f(z/L)$ has the same order as $f(z)$, and its
nonzero roots $\left\{ L\lambda_{n}\right\} _{n\in\mathbb{N}}$ would
satisfy (\ref{eq:2.10}). 

By definition (\ref{eq:2.9}) the criterion (\ref{eq:2.11}) becomes
that for all $n,k\in\mathbb{N}_{0}$,
\[
\frac{1}{L^{n+k+2}}\sum_{j=0}^{k}(-1)^{k+j}\binom{k}{j}m_{n+k-j}L^{j}=\frac{1}{L^{2}}\sum_{j=0}^{k}(-1)^{j}\binom{k}{j}\frac{m_{n+j}}{L^{n+j}}=\frac{1}{L^{n+2}}\sum_{j=1}^{\infty}\frac{1}{\lambda_{j}^{n+2}}\left(1-\frac{1}{L\lambda_{j}}\right)^{k}\ge0,
\]
which is equivalent to 
\begin{equation}
\sum_{j=0}^{k}(-1)^{j}\binom{k}{j}\frac{m_{n+j}}{L^{n+j}}\ge0,\quad\forall n,k\in\mathbb{N}_{0}.\label{eq:2.21}
\end{equation}

\end{rem}

\section{\label{sec:rh}Applications}

\subsection{Riemann $\xi(s)$ function}

Let $s=\sigma+it,\ \sigma,t\in\mathbb{R}$, the Riemann $\xi$-function
is defined by \cite{AndrewsAskeyRoy,Apostol,DLMF,Edwards}
\begin{equation}
\xi(s)=\pi^{-s/2}(s-1)\Gamma\left(1+\frac{s}{2}\right)\zeta(s),\label{eq:3.1}
\end{equation}
where $\Gamma(s)$ and $\zeta(s)$ are the respective analytic continuations
of
\begin{equation}
\Gamma(s)=\int_{0}^{\infty}e^{-x}x^{s-1}dx,\quad\sigma>0\label{eq:3.2}
\end{equation}
 and 
\begin{equation}
\zeta(s)=\sum_{n=1}^{\infty}\frac{1}{n^{s}},\quad\sigma>1.\label{eq:3.3}
\end{equation}
Then $\xi(s)$ is an order $1$ entire function that satisfies the
functional equation $\xi(s)=\xi(1-s)$, which implies the Riemann
Xi function $\Xi(s)=\xi\left(\frac{1}{2}+is\right)$ is an even entire
function of order $1$. It is well-known that all the zeros of $\Xi(s)$
are located within the proper horizontal strip $t\in(-1/2,1/2)$.
The Riemann hypothesis is equivalent to that all the zeros of $\Xi(s)$
are real. 

Since \cite{AndrewsAskeyRoy,Davenport,DLMF,Edwards}
\begin{equation}
\Xi(s)=\int_{-\infty}^{\infty}\Phi(u)e^{ius}du=2\int_{0}^{\infty}\Phi(u)\cos(us)du,\label{eq:3.4}
\end{equation}
where 
\begin{equation}
\Phi(u)=\Phi(-u)=\sum_{n=1}^{\infty}\left(4n^{4}\pi^{2}e^{9u/2}-6n^{2}\pi e^{5u/2}\right)e^{-n^{2}\pi e^{2u}}>0,\label{eq:3.5}
\end{equation}
then, 
\begin{equation}
\xi\left(\frac{1}{2}+s\right)=\sum_{n=0}^{\infty}a_{n}s^{2n},\quad a_{n}=\frac{2}{(2n)!}\int_{0}^{\infty}\Phi(u)u^{2n}du>0.\label{eq:3.6}
\end{equation}
By applying Corollary \ref{cor:4} we obtain the following:
\begin{cor}
For $k\in\mathbb{N}_{0}$ let
\begin{equation}
m_{k}=-\frac{1}{(2k+3)!\cdot2}\left(\frac{\Xi'(s)}{\Xi(s)}\right)^{(2k+3)}\bigg|_{s=0}.\label{eq:3.7}
\end{equation}
Furthermore, for $n\in\mathbb{N}_{0}$ the moment $m_{n}$ can be
expressed as a determinant 
\begin{equation}
m_{n}=(-1)^{n}\det\begin{pmatrix}1 & 0 & 0 & 0 & \dots & \frac{a_{1}^{2}}{a_{0}^{2}}-\frac{2a_{2}}{a_{0}}\\
\frac{a_{1}}{a_{0}} & 1 & 0 & 0 & \dots & \frac{a_{1}a_{2}}{a_{0}^{2}}-\frac{3a_{3}}{a_{0}}\\
\frac{a_{2}}{a_{0}} & \frac{a_{1}}{a_{0}} & 1 & 0 & \dots & \frac{a_{1}a_{3}}{a_{0}^{2}}-\frac{4a_{4}}{a_{0}}\\
\frac{a_{3}}{a_{0}} & \frac{a_{2}}{a_{0}} & \frac{a_{1}}{a_{0}} & 1 & \dots & \frac{a_{1}a_{4}}{a_{0}^{2}}-\frac{5a_{5}}{a_{0}}\\
\vdots & \vdots & \vdots & \vdots & \ddots & \vdots\\
\frac{a_{n}}{a_{0}} & \frac{a_{n-1}}{a_{0}} & \frac{a_{n-2}}{a_{0}} & \frac{a_{n-3}}{a_{0}} & \dots & \frac{a_{1}a_{n+1}}{a_{0}^{2}}-\frac{(n+2)a_{n+2}}{a_{0}}
\end{pmatrix},\label{eq:3.8}
\end{equation}
and it can be computed recursively,
\begin{equation}
(-1)^{n}m_{n}=\frac{a_{1}a_{n+1}}{a_{0}^{2}}-\frac{(n+2)a_{n+2}}{a_{0}}-\sum_{k=0}^{n-1}\frac{(-1)^{k}m_{k}a_{n-k}}{a_{0}},\label{eq:3.9}
\end{equation}
where $a_{j}$s' are defined in (\ref{eq:3.6}). 

Then the Riemann hypothesis holds if and only if there exists a positive
number $L>s_{1}^{-2}$ the sequence $\left\{ m_{k}\right\} _{k=0}^{\infty}$
satisfies 
\begin{equation}
\sum_{j=0}^{k}(-1)^{j}\binom{k}{j}\frac{m_{n+j}}{L^{n+j}}\ge0,\quad\forall n,k\in\mathbb{N}_{0},\label{eq:3.10}
\end{equation}
 where $s_{1}\approx14.1347$ is the zero of $\Xi(s)$ with smallest
positive real part.
\end{cor}

\subsection{Character $\xi(s,\chi)$ function }

For a primitive Dirichlet character $\chi(n)$ modulo $q$, let \cite{AndrewsAskeyRoy,Apostol,Davenport,Edwards}
\begin{equation}
\xi(s,\chi)=\left(\frac{q}{\pi}\right)^{(s+\kappa)/2}\Gamma\left(\frac{s+\kappa}{2}\right)L(s,\chi),\label{eq:3.11}
\end{equation}
where $\kappa$ is the parity of $\chi$ and $L(s,\chi)$ is the analytic
continuation of 
\begin{equation}
L\left(s,\chi\right)=\sum_{n=1}^{\infty}\frac{\chi(n)}{n^{s}},\quad\sigma>1.\label{eq:3.12}
\end{equation}
Then $\xi(s,\chi)$ is an entire function of order $1$ such that
\cite{Davenport}
\begin{equation}
\xi(s,\chi)=\epsilon(\chi)\xi(1-s,\overline{\chi}),\label{eq:3.13}
\end{equation}
where 
\begin{equation}
\epsilon(\chi)=\frac{\tau(\chi)}{i^{\kappa}\sqrt{q}},\quad\tau(\chi)=\sum_{n=1}^{q}\chi(n)\exp\left(\frac{2\pi in}{q}\right).\label{eq:3.14}
\end{equation}
Let 
\begin{equation}
G(s,\chi)=\xi\left(s,\chi\right)\cdot\xi\left(s,\overline{\chi}\right),\label{eq:3.15}
\end{equation}
then
\begin{equation}
G(s,\chi)=\epsilon\left(\chi\right)\cdot\epsilon\left(\overline{\chi}\right)G(1-s,\chi).\label{eq:3.16}
\end{equation}
Since \cite{Davenport}
\begin{equation}
\tau\left(\overline{\chi}\right)=\overline{\tau(\chi)},\quad\left|\tau(\chi)\right|=\sqrt{q},\label{eq:3.17}
\end{equation}
then
\begin{equation}
G(s,\chi)=G(1-s,\chi).\label{eq:3.18}
\end{equation}
Since the entire function $\xi\left(\frac{1}{2}+is,\chi\right)$ has
an integral representation, \cite{Davenport}
\begin{equation}
\xi\left(\frac{1}{2}+is,\chi\right)=\int_{-\infty}^{\infty}e^{isy}\varphi\left(y,\chi\right)dy,\label{eq:3.19}
\end{equation}
where
\begin{equation}
\varphi(y,\chi)=2\sum_{n=1}^{\infty}n^{\kappa}\chi(n)\exp\left(-\frac{n^{2}\pi}{q}e^{2y}+\left(\kappa+\frac{1}{2}\right)y\right),\label{eq:3.20}
\end{equation}
then
\begin{equation}
\xi\left(\frac{1}{2}+is,\chi\right)=\sum_{n=0}^{\infty}i^{n}a_{n}(\chi)s^{n},\label{eq:3.21}
\end{equation}
where
\begin{equation}
a_{n}(\chi)=\int_{-\infty}^{\infty}y^{n}\varphi\left(y,\chi\right)dy.\label{eq:3.22}
\end{equation}
It is known that the fast decreasing smooth function $\varphi(y,\chi)$
satisfies the functional equation \cite{Davenport}
\begin{equation}
\varphi(y,\chi)=\frac{i^{\kappa}\sqrt{q}}{\tau\left(\overline{\chi}\right)}\varphi(-y;\overline{\chi}),\quad y\in\mathbb{R},\label{eq:3.23}
\end{equation}
then for all $n\in\mathbb{N}_{0}$,
\begin{equation}
\begin{aligned} & a_{n}(\overline{\chi})=\int_{-\infty}^{\infty}y^{n}\varphi\left(y,\overline{\chi}\right)dy=(-1)^{n}\int_{-\infty}^{\infty}y^{n}\varphi\left(-y,\overline{\chi}\right)dy\\
 & =\frac{(-1)^{n}\tau\left(\overline{\chi}\right)}{i^{\kappa}\sqrt{q}}\int_{-\infty}^{\infty}y^{n}\varphi\left(y,\chi\right)dy=\frac{(-1)^{n}\tau\left(\overline{\chi}\right)}{i^{\kappa}\sqrt{q}}a_{n}(\chi).
\end{aligned}
\label{eq:3.24}
\end{equation}
Let
\begin{equation}
f(s,\chi)=s^{-2\mu}G\left(\frac{1}{2}+is,\chi\right)=s^{-2\mu}\xi\left(\frac{1}{2}+is,\chi\right)\cdot\xi\left(\frac{1}{2}+is,\overline{\chi}\right),\label{eq:3.25}
\end{equation}
where $\mu\in\mathbb{N}_{0}$ is the least nonnegative integer such
that $a_{\mu}(\chi)\neq0$, then the even entire function $f(s,\chi)$
has the series expansion
\begin{equation}
f(s,\chi)=f(-s,\chi)=\sum_{n=0}^{\infty}(-1)^{n}b_{n}(\chi)s^{2n},\label{eq:3.26}
\end{equation}
where $b_{0}(\chi)=\frac{(-1)^{\mu}\tau\left(\overline{\chi}\right)}{i^{\kappa}\sqrt{q}}a_{\mu}^{2}(\chi)$
and for $n\in\mathbb{N}$,
\begin{equation}
\begin{aligned} & b_{n}(\chi)=\sum_{j=0}^{2n}a_{j+\mu}(\chi)a_{2n-j+\mu}(\overline{\chi})\\
 & =\frac{(-1)^{\mu}\tau\left(\overline{\chi}\right)}{i^{\kappa}\sqrt{q}}\sum_{j=0}^{2n}(-1)^{j}a_{j+\mu}(\chi)a_{2n-j+\mu}(\chi).
\end{aligned}
\label{eq:3.27}
\end{equation}
It is well-known that $\xi(s,\chi)$ is an order $1$ entire function
with infinitely many zeros, all of them are in the horizontal strip
$t\in(-1/2,1/2)$, \cite{Davenport}. Then $f(s,\chi)$ is an order
$1$ even entire function with infinitely many zeros, all of them
are in the horizontal strip $t\in(-1/2,1/2)$. Clearly, all the zeros
of $\xi(s,\chi)$ on the critical line $\sigma=\frac{1}{2}$ if and
only if all the zeros of $f(s,\chi)$ are real. Therefore, the generalized
Riemann hypothesis for $L\left(s,\chi\right)$ is equivalent to that
all the zeros of $f(s,\chi)$ are real. 

By Corollary \ref{cor:4} we have the following:
\begin{cor}
Given a primitive Dirichlet character $\chi$, let $f(s,\chi)$ be
defined as in (\ref{eq:3.25}) and (\ref{eq:3.26}). For any $k\in\mathbb{N}_{0}$
let 
\begin{equation}
m_{k}(\chi)=-\frac{1}{(2k+3)!\cdot2}\left(\frac{f'(s,\chi)}{f(s,\chi)}\right)^{(2k+3)}\bigg|_{s=0}.\label{eq:3.28}
\end{equation}
For any $n\in\mathbb{N}_{0}$,
\begin{equation}
m_{n}=(-1)^{n}\det\begin{pmatrix}1 & 0 & 0 & 0 & \dots & \frac{b_{1}^{2}(\chi)}{b_{0}^{2}(\chi)}-\frac{2b_{2}(\chi)}{b_{0}(\chi)}\\
\frac{b_{1}(\chi)}{b_{0}(\chi)} & 1 & 0 & 0 & \dots & \frac{b_{1}(\chi)b_{2}(\chi)}{b_{0}^{2}(\chi)}-\frac{3b_{3}(\chi)}{b_{0}(\chi)}\\
\frac{b_{2}(\chi)}{b_{0}(\chi)} & \frac{b_{1}(\chi)}{b_{0}(\chi)} & 1 & 0 & \dots & \frac{b_{1}(\chi)b_{3}(\chi)}{b_{0}^{2}(\chi)}-\frac{4b_{4}(\chi)}{b_{0}(\chi)}\\
\frac{b_{3}(\chi)}{b_{0}(\chi)} & \frac{b_{2}(\chi)}{b_{0}(\chi)} & \frac{b_{1}(\chi)}{b_{0}(\chi)} & 1 & \dots & \frac{b_{1}(\chi)b_{4}(\chi)}{b_{0}^{2}(\chi)}-\frac{5b_{5}(\chi)}{b_{0}(\chi)}\\
\vdots & \vdots & \vdots & \vdots & \ddots & \vdots\\
\frac{b_{n}(\chi)}{b_{0}(\chi)} & \frac{b_{n-1}(\chi)}{b_{0}(\chi)} & \frac{b_{n-2}(\chi)}{b_{0}(\chi)} & \frac{b_{n-3}(\chi)}{b_{0}(\chi)} & \dots & \frac{b_{1}(\chi)b_{n+1}(\chi)}{b_{0}^{2}(\chi)}-\frac{(n+2)b_{n+2}(\chi)}{b_{0}(\chi)}
\end{pmatrix}\label{eq:3.29}
\end{equation}
and 
\begin{equation}
(-1)^{n}m_{n}(\chi)=\frac{b_{1}(\chi)b_{n+1}(\chi)}{b_{0}^{2}(\chi)}-\frac{(n+2)b_{n+2}(\chi)}{b_{0}(\chi)}-\sum_{k=0}^{n-1}\frac{(-1)^{k}m_{k}(\chi)b_{n-k}(\chi)}{b_{0}(\chi)},\quad\forall n\in\mathbb{N}_{0},\label{eq:3.30}
\end{equation}
where $b_{k}(\chi),\,k\in\mathbb{N}_{0}$ are defined in (\ref{eq:3.22})
and (\ref{eq:3.27}). 

If 
\begin{equation}
\frac{b_{n}(\chi)}{b_{0}(\chi)}=\sum_{j=0}^{2n}(-1)^{j}\frac{a_{j+\mu}(\chi)}{a_{\mu}(\chi)}\frac{a_{2n-j+\mu}(\chi)}{a_{\mu}(\chi)}>0,\quad\forall n\in\mathbb{N},\label{eq:3.31}
\end{equation}
then the generalized Riemann hypothesis for $L\left(s,\chi\right)$
holds if and only if there exists a positive number $L>s_{1}^{-2}(\chi)$
such that
\begin{equation}
\sum_{j=0}^{k}(-1)^{j}\binom{k}{j}\frac{m_{n+j}(\chi)}{L^{n+j}}\ge0,\quad\forall n,k\in\mathbb{N}_{0},\label{eq:3.32}
\end{equation}
 where $s_{1}(\chi)>0$ is the  zero of $f(s,\chi)$ with the smallest
positive real part. 
\end{cor}


\begin{thebibliography}{1}
\bibitem{AndrewsAskeyRoy}G. Andrews, R. Askey and R. Roy, \emph{Special
Functions}, 1st edition, Cambridge University Press.

\bibitem{Apostol}T. Apostol, \emph{Introduction to Analytic Number
Theory}, Springer, New York, 2010.

\bibitem{Boas}R. P. Boas, \emph{Entire Functions}, 1st edition, Academic
Press, 1954.

\bibitem{Davenport}H. Davenport, Multiplicative Number Theory, Springer-Verlag,
New York, 1980.

\bibitem{DLMF}NIST Digital Library of Mathematical Functions. http://dlmf.nist.gov/,
Release 1.1.3 of 2021-09-15. F. W. J. Olver, A. B. Olde Daalhuis,
D. W. Lozier, B. I. Schneider, R. F. Boisvert, C. W. Clark, B. R.
Miller, B. V. Saunders, H. S. Cohl, and M. A. McClain, eds.

\bibitem{Edwards}H. M. Edwards, \emph{Riemann's Zeta Function}, Dover
Publications, 2001.

\bibitem{Ismail}M. E. H. Ismail, \emph{Classical and Quantum Orthogonal
Polynomials in One Variable, }Cambridge University Press, Cambridge,
2005.

\bibitem{RoydenFitzpatrick}R. Royden and P. Fitzpatrick, \emph{Real
Analysis}, 4th edition, Pearson Education, 2010.

\bibitem{ShohatTamarkin}J. A. Shohat and J. D. Tamarkin, The Problem
of Moments, American mathematical society, New York, 1943. 

\end{thebibliography}
\end{document}